\documentclass[reqno]{amsart}
\usepackage[all]{xy}

\usepackage[mathscr]{eucal} 

\numberwithin{equation}{section}

\newtheorem{Theorem}{Theorem}[section]
\newtheorem{Prop}[Theorem]{Proposition}
\newtheorem{Lemma}[Theorem]{Lemma}
\newtheorem{Corollary}[Theorem]{Corollary}
\theoremstyle{definition}
\newtheorem{Definition}[Theorem]{Definition}
\newtheorem{Example}[Theorem]{Example}
\newtheorem{Remark}[Theorem]{Remark}

\newcommand{\Z}{\mathbb{Z}}
\newcommand{\PP}{\mathbb{P}}
\newcommand{\Q}{\mathbb{Q}}

\begin{document}

\bibliographystyle{amsalpha}

\title[Difference equations and cluster algebras I]{
Difference equations and cluster algebras I:
Poisson bracket for integrable difference equations
}

\author[R.\  Inoue]{Rei Inoue}
\address{ Rei  Inoue:
Faculty of Pharmaceutical Sciences, Suzuka University of Medical Science,
Suzuka, 513-8670, Japan}

\author{Tomoki Nakanishi}
\address{
Tomoki Nakanishi: Graduate School of Mathematics, Nagoya University,
Nagoya, 464-8604, Japan}

\maketitle

\begin{abstract}
We introduce the cluster algebraic formulation of the integrable difference 
equations,
the discrete Lotka-Volterra equation and the discrete Liouville equation,
from the view point of the general T-system and Y-system.
We also study the Poisson structure for the cluster algebra,
and give the associated Poisson bracket for the two difference equations.
\end{abstract}


\section{Introduction}

The T-systems and Y-systems are difference equations
arising from the study of various integrable statistical and field theoretical
models in 1990's. See  \cite{KNS10} for a recent review of the subject.
Since the introduction of cluster algebras by Fomin and Zelevinsky
around 2000, it has been gradually noticed 
that these systems are naturally formulated
with cluster algebras
\cite{FZ03, HL09, Ke10, DiFrancesco09a, Ke2, IIKNS10, KNS09, 
Inoue10a, Inoue10b, Nakanishi10a, Nakanishi10b};
furthermore, their cluster algebraic nature
is essential to prove the long-standing conjectures on their
periodicities and the associated dilogarithm identities
\cite{FZ03, Ke10, Ke2, IIKNS10, Nakanishi09, Inoue10a,
Inoue10b, Nakanishi10b, N10}.
More recently, by inverting the point of view,
T-systems and Y-systems are extensively generalized
so that they are associated with
any
periodic sequence of exchange matrices in a cluster algebra \cite{N10}.
This generalization includes the difference equations
studied earlier in \cite{FZ07,FordyMarsh09} as special cases.
In this paper and the subsequent ones, we are going to 
study these general T and Y-systems, especially 
in connection with known {\em integrable} difference
equations.

Let us give several reasons/motivations why we are interested
in such difference equations arising from cluster algebras.

(i) They provide infinitely many difference equations, some of
which are known integrable difference equations
(Hirota-Miwa \cite{HL09, DiFrancesco09a,IIKNS10}, Toda \cite{GSV09},
Somos 4 \cite{FZ02,Hone07,FordyMarsh09},
discrete Liouville,  discrete Lotka-Volterra equations, {\em etc.}),
and almost all of which are new ones.
Therefore, they might provide the ground for a unified treatment of
a wide variety of (known and unknown) integrable difference equations.

(ii) They have the built-in Poisson and symplectic structures
\cite{GSV02,GSV03,GSV09,GSV10, FockGon03, FockGon07, Fordy10}.
We would think that 
the Poisson structure for integrable {\it difference} equations 
are not understood enough yet,
comparing with that for integrable {\it differential} equations.
We expect that they provide some key to this problem.

(iii) Any Y-system and the corresponding T-system (the latter  is
often the equation for the {\em $\tau$ function\/})
are unified by a cluster algebra,
and, {\em they are formally solved from the beginning\/} through
the {\em categorification\/} of the cluster algebra,
the cluster category, recently developed by Keller and others
\cite{Caldero06,Buan06, DK,FK, Ke10, A, Ke2,Plamondon10a,Plamondon10b}.
Furthermore, the both systems reduce to the {\em tropical Y-system},
which is  a much simpler piecewise-linear system.
We call the totality of
these phenomena the {\em integrability by categorification}.
These results and methods are not limited to bilinear equations.
A more account will be given in Section \ref{subsec:categorification}.

The aim of this paper is to give the cluster algebraic 
formulation and the associated Poisson bracket
for two typical examples of integrable difference equations,
the discrete Lotka-Volterra equation (discrete LV equation) \cite{HT94,HT95}
and the discrete Liouville equation \cite[\S 3]{FV99}.
We follow \cite{FZ07} for the definition of 
the cluster algebra $\mathcal{A}(B,x,y)$ given by 
a skew-symmetrizable matrix $B$, the cluster $x$ and 
the coefficient tuple $y$, with the mutations. 
For the purpose, we study
the {\it mutation compatible} Poisson bracket
and define the {\it Poisson structure} for $\mathcal{A}(B,x,y)$,
which corresponds to a generalization of those in \cite{GSV02,FockGon07}.
We are especially interested in the case that 
the matrix $B$ is infinite and not invertible,
since the discrete LV equation is such case.
The discrete periodic Liouville equation is 
an example of the case that the matrix $B$ is finite and invertible.
We further study the Poisson bracket {\it with symmetry}
for the discrete LV equation.

This paper is organized as follows:
in \S 2, we briefly explain basic definitions of cluster algebras,
and formulate the discrete LV equation and the discrete Liouville equation
in the framework of the cluster algebras.
The notion of the integrability by categorification is explained
in \S 2.4. 
In \S 3 and \S 4, we construct the mutation compatible Poisson bracket
(Definition \ref{def:mut-com}) at Theorem \ref{prop:x-Poisson:tn}
and Proposition \ref{prop:xy-poisson},
and define the Poisson structure for the cluster algebra.
Especially, the Poisson matrix $P$ is formulated 
at Theorem \ref{th:x-Poisson}
(resp. Theorem \ref{th:x-Poisson-infinite})
for a finite $B$ (resp. an infinite $B$). 
Finally in \S 5, we apply \S 3 and \S 4 to \S 2, 
and study the Poisson brackets for the two integrable difference equations.

\subsection*{Acknowledgements}

R.~I. is partially supported by Grant-in-Aid for Young Scientists (B)
(22740111).

\section{Cluster algebraic formulation of difference equations} 
\subsection{Cluster algebra: basic definitions}
\label{subsec:def}

We briefly explain basic definitions of cluster algebra 
following \cite{FZ07}.
Let $I \subset \Z$ be an index set. (It can be infinite.)
We say that an integer matrix $B = (b_{ij})_{i,j \in I}$ 
is {\it skew-symmetrizable},
if there is a diagonal positive integer matrix 
$D= \mathrm{diag}(d_i)_{i \in I}$ 
such that $DB$ is skew-symmetric. 
When $I$ is infinite, we always assume that a skew-symmetrizable matrix $B$ 
has only finitely many nonzero elements in each row and in each column.

For a skew-symmetrizable matrix $B$ and $k \in I$, 
we have the mutation $B' = \mu_k(B)$ of $B$ at $k$ defined by 
\begin{align}\label{B-mutation}
  b'_{ij} 
  = 
  \begin{cases}
    -b_{ij} & \text{$i=k$ or $j=k$}, \\
    b_{ij} + \frac{1}{2}(|b_{ik}| b_{kj} + b_{ik}|b_{kj}|)        
    & \text{otherwise}.
  \end{cases} 
\end{align}
The matrix $B'$ is again skew-symmetrizable.

Let $\PP$ be a given semifield and write $\Q \PP$ for the quotient field of 
the group ring $\Z \PP$ of $\PP$.
For an $I$-tuple $y=(y_i)_{i \in I}$ as $y_i \in \PP$, 
the mutation $y' = \mu_k(y)$ of $y$ is defined by the exchange relation
\begin{align}\label{y-mutation}
  y_i' 
  = 
  \begin{cases} 
  y_k^{-1} & i= k, \\
  \displaystyle{
  y_i \left(\frac{y_k}{1 \oplus y_k} \right)^{b_{ki}}} 
  & i \neq k, b_{ki} \geq 0, \\
  y_i (1 \oplus y_k)^{-b_{ki}} & i \neq k, b_{ki} \leq 0.
  \end{cases}
\end{align}
Let $\Q \PP(u)$ be the rational functional field of algebraically independent
variables $\{u_i\}_{i \in I}$.
For an $I$-tuple $x=(x_i)_{i \in I}$ such that $\{x_i\}_{i \in I}$
is a free generating set of $\Q\PP(u)$ and for $k \in I$,
the mutation $x' = \mu_k(x)$ of $x$ is defined by the exchange relation
\begin{align}\label{x-mutation}
  x'_i 
  =
  \begin{cases}
    x_i & i \neq k, \\
    \displaystyle{
    \frac{y_k \prod_{j:b_{jk} > 0} x_j^{b_{jk}} 
          + \prod_{j:b_{jk} < 0} x_j^{-b_{jk}}}{(1 \oplus y_k)x_k}} 
    & i = k. 
  \end{cases}
\end{align} 
The mutation \eqref{B-mutation}--\eqref{x-mutation} is involutive,
{\it i.e.}, $\mu_k^2=\mathrm{id}$.
The $I$-tuples $x$ and $y$ are respectively called a cluster 
and a coefficients tuple,
and $x_i$ and $y_i$ are respectively called a cluster variable
and a coefficient.
By iterating mutations by starting with the initial seed $(B,x,y)$,
we obtain seeds $(B',x',y')$.
The cluster algebra $\mathcal{A}(B,x,y)$ is 
a $\Z \PP$-subalgebra of $\Q \PP(u)$
generated by all the cluster variables in all the seeds.

Alternatively, one may consider the {\em $I$-regular tree} $\mathbb{T}_I$
whose edges are labeled by $I$ with a distinguished vertex 
$t_0\in \mathbb{T}_I$
(the {\em initial vertex}),
and regard that a seed $(B',x',y')$ is assigned to each vertex $t' \in \mathbb{T}_I$
so that (i) for the vertex $t_0$ the initial seed $(B,x,y)$ is attached,
and (ii) for any edge $t'\, \frac{k}{\phantom{aaa}}\,  t''$ the corresponding
seeds are related by the mutation at $k$.
We call the assignment  the {\em cluster pattern} for the cluster algebra
$\mathcal{A}(B,x,y)$.

In the rest of this section, we let $\PP$ be the universal semifield
$\PP_{\mathrm{univ}}(y)$ generated by $y=(y_i)_{i \in I}$,
which is the set of all rational functions of $y_i ~(i \in I)$
written as subtraction-free expressions.
Here the operation $\oplus$ is the usual addition.

\subsection{Discrete Lotka-Volterra equation}
\label{subsec:dLV}

The {\em discrete Lotka-Volterra equation} (discrete LV equation) 
is the difference equation \cite{HT94,HT95}: 
\begin{align}
\label{eq:U1}
u^{t+1}_{n+1}=u^t_n \frac{1+\delta u^{t+1}_n}
{1+\delta u^t_{n+1}},
\end{align}
where $u_n^t$ is a function of $(n,t) \in \Z^2$. 
This has the bilinear form in the following sense:
suppose that $\{\tau_n^t ~|~ (n,t) \in \Z^2 \}$ satisfies the relation
(the bilinear form of the discrete LV equation \cite{HT94,HT95}):
\begin{align}
\label{eq:T1}
\tau^{t-1}_n\tau^{t+1}_{n+1}
=
\frac{\delta}{1+\delta} \tau^{t-1}_{n+1}\tau^{t+1}_n
+
\frac{1}{1+\delta} \tau^{t}_{n}\tau^{t}_{n+1}.
\end{align}
Then, $\{u_n^t \}$ defined by
\begin{align}\label{u-tau}
u^{t}_n&=\frac{\tau^{t+1}_n\tau^{t-1}_{n+1}}
{\tau^t_{n+1}\tau^t_n}
\end{align}
satisfies \eqref{eq:U1}. 
Note that not all the solutions to \eqref{eq:U1} are written in this way.
Here we concentrate on the solutions of \eqref{eq:U1} admitting 
the bilinear form \eqref{eq:T1}.

Let $Q=Q(0)$ be the infinite quiver depicted at Figure 1,
where the vertex set of $Q$ is labelled by $I = \{i ~|~ i \in \Z \}$.
Let $I_{\overline{k}} := \{i ~|~ i \in 3 \Z + k \} ~(k=0,1,2)$.
We identify the quiver $Q$ without loops and $2$-cycles
with the skew-symmetric matrix 
$B = (b_{ij})_{i,j \in I}$ in 
the standard way. Namely, we set $b_{ij} = -b_{ji} = t$ for $i \neq j$
if there are $t$ arrows from the vertex $i$ to the vertex $j$ in $Q$, and 
$b_{ij} = 0$ otherwise.
Then the corresponding matrix $B = (b_{ij})_{i,j \in I}$ is 
skew-symmetric and written as
\begin{align}\label{B-LV}
  \begin{split}
  &b_{3k,3k+i} = \delta_{i,1} + \delta_{i,-1} - \delta_{i,2} - \delta_{i,-2},
  \\
  &b_{3k+1,3k+1+i} = \delta_{i,1} + \delta_{i,2} + \delta_{i,-3} 
               - \delta_{i,-1} - \delta_{i,-2} - \delta_{i,3},
  \\
  &b_{3k+2,3k+2+i} = 
  - \delta_{i,1} - \delta_{i,-1} + \delta_{i,2} + \delta_{i,-2},
  \end{split}   
\end{align}
for $i,k \in \Z$.
Note that $B$ is $3$-periodic, {\it i.e.}, $b_{i+3,j+3} = b_{i,j}$.
Let $\mathcal{A}(B,x,y)$ be the corresponding cluster algebra.

Define the composite mutation  
$\mu_{\overline{k}} = \prod_{i \in I_{\overline{k}}} \mu_i$ for $k=0,1,2$.
We define $B(u) = (b(u)_{ij})_{i,j \in I}$ 
for $u \in \Z$ by
\begin{align}\label{Bu-LV}
  \begin{split}
  &B(0) = B, 
  \\
  &B(3i+k+1) = \mu_{\overline{k}}(B(3i+k)) \quad i \geq 0,
  \\
  &B(3i+k) = \mu_{\overline{k}}(B(3i+k+1)) \quad i \leq -1,
  \end{split}
\end{align} 
for $k=0,1,2$.
We have the following symmetry and periodicity of $B(u)$, which is
easy to see by Figure 1.
\begin{Lemma}
We have 
\begin{align}
\label{eq:b11}
b(u)_{i+3,j+3}&=b(u)_{i,j},\\
\label{eq:b12}
b(u+1)_{i,j}&=b(u)_{i-1,j-1},\\
\label{eq:b13}
b(u+3)_{i,j}&=b(u)_{i,j}.
\end{align}
\end{Lemma}

In particular, \eqref{eq:b13} means that
the concatenation of the sequence
$(\ldots, -3, 0, 3, \ldots)$, $(\ldots, -2, 1, 4,\ldots)$, 
$(\ldots, -1, 2, 5, \ldots)$
is a natural infinite analogue of a {\em regular period} of 
$B=B(0)$ in the terminology of \cite{N10}.
(Also, the property \eqref{eq:b12} is a natural infinite analogue of the 
{\it mutation periodicity} of $B$ in the terminology of \cite{FordyMarsh09}.)

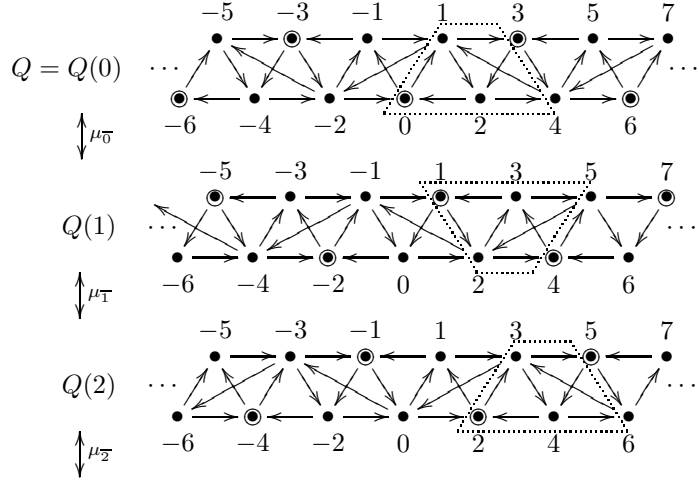
\begin{figure}\label{fig:Q-cluster}
\begin{math}
%
\begin{xy}
(0,0)*{\bullet},+/d10pt/*{ -6},
(5,8)*{\bullet},+/u10pt/*{-5},
(10,0)*{\bullet},+/d10pt/*{-4},
(15,8)*{\bullet},+/u10pt/*{-3},
(20,0)*{\bullet},+/d10pt/*{-2},
(25,8)*{\bullet},+/u10pt/*{-1},
(30,0)*{\bullet},+/d10pt/*{0},
(35,8)*{\bullet},+/u10pt/*{1},
(40,0)*{\bullet},+/d10pt/*{2},
(45,8)*{\bullet},+/u10pt/*{3},
(50,0)*{\bullet},+/d10pt/*{4},
(55,8)*{\bullet},+/u10pt/*{5},
(60,0)*{\bullet},+/d10pt/*{6},
(65,8)*{\bullet},+/u10pt/*{7},
(0,0)*\cir<3pt>{},
(15,8)*\cir<3pt>{},
(30,0)*\cir<3pt>{},
(45,8)*\cir<3pt>{},
(60,0)*\cir<3pt>{},
%
(-2,4)*{\cdots},
(67,4)*{\cdots},
%
(-15,4)*{Q=Q(0)},
\ar (1,1.6);(4,6.4)
\ar (21,1.6);(24,6.4)
\ar (31,1.6);(34,6.4)
\ar (51,1.6);(54,6.4)
\ar (61,1.6);(64,6.4)
\ar (14,6.4);(11,1.6)
\ar (44,6.4);(41,1.6)
\ar (6,6.4);(9,1.6)
\ar (16,6.4);(19,1.6)
\ar (36,6.4);(39,1.6)
\ar (46,6.4);(49,1.6)
\ar (29,1.6);(26,6.4)
\ar (59,1.6);(56,6.4)
\ar (8,0);(2,0)
\ar (23,8);(17,8)
\ar (38,0);(32,0)
\ar (53,8);(47,8)
\ar (7,8);(13,8)
\ar (12,0);(18,0)
\ar (22,0);(28,0)
\ar (27,8);(33,8)
\ar (37,8);(43,8)
\ar (42,0);(48,0)
\ar (52,0);(58,0)
\ar (57,8);(63,8)
\ar (18,1);(7,7)
\ar (48,1);(37,7)
\ar (33,7);(22,1)
\ar (63,7);(52,1)
\ar@{.}(27.25,-2);(34.75,10) 
\ar@{.}(49.75,-2);(42.25,10) 
\ar@{.}(27.25,-2);(49.75,-2)
\ar@{.}(34.75,10);(42.25,10)
\ar@{<->}^{\mu_{\overline{0}}} (-13,-2);(-13,-8)
\end{xy}
\end{math}
%
\begin{math}
\begin{xy}
(0,0)*{\bullet},+/d10pt/*{ -6},
(5,8)*{\bullet},+/u10pt/*{-5},
(10,0)*{\bullet},+/d10pt/*{-4},
(15,8)*{\bullet},+/u10pt/*{-3},
(20,0)*{\bullet},+/d10pt/*{-2},
(25,8)*{\bullet},+/u10pt/*{-1},
(30,0)*{\bullet},+/d10pt/*{0},
(35,8)*{\bullet},+/u10pt/*{1},
(40,0)*{\bullet},+/d10pt/*{2},
(45,8)*{\bullet},+/u10pt/*{3},
(50,0)*{\bullet},+/d10pt/*{4},
(55,8)*{\bullet},+/u10pt/*{5},
(60,0)*{\bullet},+/d10pt/*{6},
(65,8)*{\bullet},+/u10pt/*{7},
(5,8)*\cir<3pt>{},
(20,0)*\cir<3pt>{},
(35,8)*\cir<3pt>{},
(50,0)*\cir<3pt>{},
(65,8)*\cir<3pt>{},
%
(-2,4)*{\cdots},
(67,4)*{\cdots},
%
(-15,4)*{\phantom{Q=}Q(1)},
\ar (11,1.6);(14,6.4)
\ar (21,1.6);(24,6.4)
\ar (41,1.6);(44,6.4)
\ar (51,1.6);(54,6.4)
\ar (4,6.4);(1,1.6)
\ar (34,6.4);(31,1.6)
\ar (64,6.4);(61,1.6)
\ar (6,6.4);(9,1.6)
\ar (26,6.4);(29,1.6)
\ar (36,6.4);(39,1.6)
\ar (56,6.4);(59,1.6)
\ar (19,1.6);(16,6.4)
\ar (49,1.6);(46,6.4)

\ar (13,8);(7,8)
\ar (28,0);(22,0)
\ar (43,8);(37,8)
\ar (58,0);(52,0)
\ar (2,0);(8,0)
\ar (12,0);(18,0)
\ar (17,8);(23,8)
\ar (27,8);(33,8)
\ar (32,0);(38,0)
\ar (42,0);(48,0)
\ar (47,8);(53,8)
\ar (57,8);(63,8)
\ar (8,1);(-3,7)
\ar (38,1);(27,7)
\ar (23,7);(12,1)
\ar (53,7);(42,1)
\ar@{.}(47.25,-2);(54.75,10)
\ar@{.}(39.75,-2);(32.25,10)
\ar@{.}(39.75,-2);(47.25,-2)
\ar@{.}(32.25,10);(54.75,10)
\ar@{<->}^{\mu_{\overline{1}}} (-13,-2);(-13,-8)
\end{xy}
\end{math}
%
\begin{math}
\begin{xy}
(0,0)*{\bullet},+/d10pt/*{ -6},
(5,8)*{\bullet},+/u10pt/*{-5},
(10,0)*{\bullet},+/d10pt/*{-4},
(15,8)*{\bullet},+/u10pt/*{-3},
(20,0)*{\bullet},+/d10pt/*{-2},
(25,8)*{\bullet},+/u10pt/*{-1},
(30,0)*{\bullet},+/d10pt/*{0},
(35,8)*{\bullet},+/u10pt/*{1},
(40,0)*{\bullet},+/d10pt/*{2},
(45,8)*{\bullet},+/u10pt/*{3},
(50,0)*{\bullet},+/d10pt/*{4},
(55,8)*{\bullet},+/u10pt/*{5},
(60,0)*{\bullet},+/d10pt/*{6},
(65,8)*{\bullet},+/u10pt/*{7},
%
(10,0)*\cir<3pt>{},
(25,8)*\cir<3pt>{},
(40,0)*\cir<3pt>{},
(55,8)*\cir<3pt>{},
%
(-2,4)*{\cdots},
(67,4)*{\cdots},
%
(-15,4)*{\phantom{Q=}Q(2)},
\ar (1,1.6);(4,6.4)
\ar (11,1.6);(14,6.4)
\ar (31,1.6);(34,6.4)
\ar (41,1.6);(44,6.4)
\ar (61,1.6);(64,6.4)
\ar (24,6.4);(21,1.6)
\ar (54,6.4);(51,1.6)
\ar (16,6.4);(19,1.6)
\ar (26,6.4);(29,1.6)
\ar (46,6.4);(49,1.6)
\ar (56,6.4);(59,1.6)
\ar (9,1.6);(6,6.4)
\ar (39,1.6);(36,6.4)
\ar (18,0);(12,0)
\ar (33,8);(27,8)
\ar (48,0);(42,0)
\ar (63,8);(57,8)
\ar (2,0);(8,0)
\ar (7,8);(13,8)
\ar (17,8);(23,8)
\ar (22,0);(28,0)
\ar (32,0);(38,0)
\ar (37,8);(43,8)
\ar (47,8);(53,8)
\ar (52,0);(58,0)
\ar (28,1);(17,7)
\ar (58,1);(47,7)
\ar (13,7);(2,1)
\ar (43,7);(32,1)
\ar@{.}(37.25,-2);(44.75,10)
\ar@{.}(59.75,-2);(52.25,10)
\ar@{.}(37.25,-2);(59.75,-2)
\ar@{.}(44.75,10);(52.25,10)
\ar@{<->}^{\mu_{\overline{2}}} (-13,-2);(-13,-8)
\end{xy}
\end{math}
\caption{Quiver $Q$. The encircled vertices are the forward mutation points. 
}
\end{figure}

In general, with such a periodicity of the matrix $B$
one can associate the T-system and Y-system \cite[Proposition 5.11]{N10} 
as follows:
in the same manner as $B(u)$, we define $I$-tuples $x(u)$ and $y(u)$  
for $u \in \Z$ by 
\begin{align}\label{LV-seq}
\begin{split}
\cdots
&\stackrel{\mu_{\overline{1}}}{\longleftrightarrow} (B(-1),x(-1),y(-1)) 
\stackrel{\mu_{\overline{2}}}{\longleftrightarrow} (B(0),x(0),y(0))
\\
&\stackrel{\mu_{\overline{0}}}{\longleftrightarrow} (B(1),x(1),y(1))
\stackrel{\mu_{\overline{1}}}{\longleftrightarrow} (B(2),x(2),y(2))
\stackrel{\mu_{\overline{2}}}{\longleftrightarrow} (B(3),x(3),y(3)) 
\stackrel{\mu_{\overline{0}}}{\longleftrightarrow} \cdots
\end{split}
\end{align}
Strictly speaking, $(B(u),x(u),y(u))$ is not a seed of $\mathcal{A} 
(B,x,y)$,
since it is obtained only through an {\em infinite\/} sequence of  
mutations.
However, it is well defined; furthermore each
$x_i(u)$ and $y_i(u)$ are respectively a cluster variable
and a coefficient of $\mathcal{A}(B,x,y)$,
because they are obtained through some {\em finite} subsequence of  
mutations in \eqref{LV-seq}
due to the locality of the exchange relations.

Let $P_+ = \{(u,i) \in \Z^2 ~|~ u \equiv i \mod 3 \}$ be 
the set of the {\it forward mutation points} as depicted in Figure 1.
Then, thanks to \eqref{y-mutation} and \eqref{x-mutation},
we have the following relations among $x_i(u)$ and $y_i(u)$ with 
$(u,i) \in P_+$:
\begin{align}
\label{x-rel}
  &x_i(u) x_i(u+3)
  =
    \frac{y_i(u) x_{i-2}(u+1)x_{i+2}(u+2) + x_{i-1}(u+2)x_{i+1}(u+1)}
       {1+ y_i(u)},
  \\
\label{y-rel}
  &y_i\left(u \right)y_i\left(u+3 \right)
  =
  \frac{
  \displaystyle
  (1+y_{i-2}(u+1))(1+y_{i+2}(u+2)) 
  }
  {
  \displaystyle
  (1+y_{i+1}(u+1)^{-1})(1+y_{i-1}(u+2)^{-1})
  }.
\end{align}
The relations \eqref{x-rel} and \eqref{y-rel} are respectively called
the T-system and the Y-system for the sequence \eqref{Bu-LV}.
Further, by following \cite[Proposition 3.9]{FZ07} we define
$\hat{y}_i(u)$ by
\begin{align}\label{yhat}
\hat{y}_i(u) 
= y_i(u) \frac{x_{i-2}(u+1)x_{i+2}(u+2)}{x_{i-1}(u+2)x_{i+1}(u+1)},
\end{align}
for $(u,i) \in P_+$. Then $\hat{y}_i(u)$ again satisfies the relation 
\eqref{y-rel}, {\it i.e.},
\begin{align}
\label{yhat-rel}
  &\hat{y}_i\left(u \right)\hat{y}_i\left(u+3 \right)
  =
  \frac{
  \displaystyle
  (1+\hat{y}_{i-2}(u+1))(1+\hat{y}_{i+2}(u+2)) 
  }
  {
  \displaystyle
  (1+\hat{y}_{i+1}(u+1)^{-1})(1+\hat{y}_{i-1}(u+2)^{-1})
  }.
\end{align}
Note that \eqref{yhat-rel} is equivalent to 
\begin{align}\label{yhat-rel22}
\frac{\hat{y}_{i-1}\left(u+2 \right)}
{\hat{y}_i\left(u \right)}
\frac{1+\hat{y}_{i-2}(u+1)}
{1+\hat{y}_{i+1}(u+1)}
=
\frac{\hat{y}_{i}\left(u+3 \right)}
{\hat{y}_{i+1}\left(u+1 \right)}
\frac{1+\hat{y}_{i-1}(u+2)}
{1+\hat{y}_{i+2}(u+2)}.
\end{align} 
   
When we take the constant solution $y_i(u)=\delta ~(\delta \in \Q)$
of \eqref{y-rel},
\eqref{x-rel} reduces to 
\begin{align}\label{rel-x2}
x_i(u)x_i(u+3)
=
\frac{\delta x_{i-2}(u+1)x_{i+2}(u+2) + x_{i-1}(u+2)x_{i+1}(u+1)}
     {1+\delta},
\end{align}
and \eqref{yhat-rel22} reduces to
\begin{align}\label{yhat-rel2}
\frac{\hat{y}_{i-1}\left(u+2 \right)}
{\hat{y}_i\left(u \right)}
\frac{1+\hat{y}_{i-2}(u+1)}
{1+\hat{y}_{i+1}(u+1)}
=
\frac{\hat{y}_{i}\left(u+3 \right)}
{\hat{y}_{i+1}\left(u+1 \right)}
\frac{1+\hat{y}_{i-1}(u+2)}
{1+\hat{y}_{i+2}(u+2)}
= 1.
\end{align} 
Via the identification
$$
  x_i(u) = \tau_n^t, \quad \hat{y}_i(u) = \delta u_n^t,
$$  
with the coordinate transformation 
\begin{align}\label{tn-ui}
  t=\frac{1}{3}(2u+i), \quad n=\frac{1}{3}(u-i),
\end{align}
we see that \eqref{rel-x2}, \eqref{yhat} and \eqref{yhat-rel2} 
are nothing but \eqref{eq:T1}, \eqref{u-tau} and \eqref{eq:U1} respectively.

\subsection{Discrete Liouville equation}
\label{subsec:d-Liu}

Fix $N \in \Z_{>1}$.
The $N$-periodic {\em discrete Liouville equation} 
is given by \cite{FV99, FKV01}
\begin{align}\label{d-Liu}
\chi_{n,t+1} \chi_{n,t-1} = (1+\chi_{n-1,t})(1+\chi_{n+1,t}),
\end{align}
where $\chi_{n,t}$ is a function of $(n,t) \in (\Z / N \Z,\Z)$.

When $N=2m$, this equation is formulated by the cluster algebra
of type $A_{2m-1}^{(1)}$ as follows \cite{FZ07}.
Let $Q$ be the quiver of the Dynkin diagram of type $A_{2m-1}^{(1)}$
in Figure 2 (a),
and let $I = \{0,1,\ldots,2m-1\}$ be the index set of $Q$.
The quiver $Q$ is bipartite,
and we set $I_+ = \{i \in I:\mathrm{even}\}$, 
$I_- = \{i \in I:\mathrm{odd}\}$. 
The corresponding matrix $B=(b_{ij})_{i,j \in I}$ 
is skew-symmetric and given by 
\begin{align}
  b_{2k,i} = - \delta_{2k-1,i} - \delta_{2k+1,i}, 
  \quad 
  b_{2k+1,i} =  \delta_{2k,i} + \delta_{2k+2,i},
\end{align}
where the indices $i,j$ of $b_{i,j}$ is in $\Z / 2m \Z$.
This $B$ is $2$-periodic, {\it i.e.}, $b_{i+2,j+2} = b_{i,j}$.

By using the composite mutations $\mu_+ = \prod_{i \in I_+} \mu_i$
and $\mu_- = \prod_{i \in I_-} \mu_i$, we define seed
$(B(u), x(u),y(u))$ for $u \in \Z$ by
\begin{align}\label{Liu-Bxy}
\begin{split}
\cdots
&\stackrel{\mu_{+}}{\longleftrightarrow} (B(-1),x(-1),y(-1)) 
\stackrel{\mu_{-}}{\longleftrightarrow} (B(0),x(0),y(0))
\\
&\stackrel{\mu_{+}}{\longleftrightarrow} (B(1),x(1),y(1))
\stackrel{\mu_{-}}{\longleftrightarrow} (B(2),x(2),y(2))
\stackrel{\mu_{+}}{\longleftrightarrow} \cdots
\end{split}
\end{align}
by starting $B(0) = B$, $x(0) = x$ and $y(0) = y$.
We have the periodicity of $B(u)$ as $B(u+2) = B(u)$.
Let $P_+ = \{(u,i) ~|~ i+u:\text{even} \}$ be the set of forward mutation
points. Again, one can associate the T- and Y-systems for $x_i(u)$ and 
$y_i(u)$ with $(u,i) \in P_+$.
Then the exchange relations \eqref{x-mutation} and \eqref{y-mutation} become
\begin{align}
  &x_i(u+2) x_i(u) = \frac{y_i(u)x_{i+1}(u+1) x_{i-1}(u+1)+1}{1+y_{i}(u)},
  \\
  \label{Liu-y-mutation}
  &y_i(u+2) y_i(u) = (1+y_{i+1}(u+1))(1+y_{i-1}(u+1)),
\end{align}
for $i \in \Z/2m \Z$.
One sees that \eqref{Liu-y-mutation} is nothing but \eqref{d-Liu}
via the identification $y_i(u) = \chi_{i,u}$.

When $N=2m+1$, the equation is formulated by the cluster algebra
of type $A_{4m+1}^{(1)}$ \cite[\S 6.4.2]{KNS09}.
(One may naturally think the cluster algebra of type $A_{2m}^{(1)}$,
but it does not work because the quiver of type $A_{2m}^{(1)}$
is not bipartite.)
Let $Q$ be the quiver of the Dynkin diagram $Q$ of type $A_{4m+1}^{(1)}$
as Figure 2 (b).
Let $I = \{0_+,1_+,\ldots,2m_+,0_-,1_-,\ldots,2m_-\}$ 
be an index set of $Q$,
and set $I_+ = \{0_+,1_+,\ldots, 2m_+\}$, $I_- = \{0_-,1_-, \ldots,2m_-\}$.
Then the corresponding matrix $B = (b_{ij})_{i,j \in I}$ is given by
\begin{align*}
  b_{k_+,i} = - \delta_{k-1_-,i} - \delta_{k+1_-,i}, 
  \quad 
  b_{k_-,i} =  \delta_{k+1_+,i} + \delta_{k-1_+,i}.
\end{align*}

By using the composite mutations $\mu_+ = \prod_{i \in I_+} \mu_i$
and $\mu_- = \prod_{i \in I_-} \mu_i$, we define seeds $(B(u),x(u),y(u))$
for $u \in \Z$ in the same way as \eqref{Liu-Bxy}. 
Then we obtain the exchange relations:
\begin{align}
  &x_{i_\pm}(u+2) x_{i_\pm}(u) = 
  \frac{y_{i_\pm}(u)x_{i+1_\mp}(u+1) x_{i-1_\mp}(u+1)+1}{1+y_{i_\pm}(u)},
  \\
  \label{Liu-y-mutationb}
  &y_{i_\pm}(u+2) y_{i_\pm}(u) = (1+y_{i+1_\mp}(u+1))(1+y_{i-1_\mp}(u+1)),
\end{align}
for $i \in \Z / (2m+1) \Z$.
One sees that \eqref{Liu-y-mutationb} becomes \eqref{d-Liu}
via the identification 
$y_{i_+}(2u)=\chi_{i,2u}, ~y_{i_-}(2u+1) = \chi_{i,2u+1}$
for $(u,i) \in \Z \times \{0,1,\ldots,2m\}$.

\begin{figure}
\label{fig:Q-cluster-Liu}
\begin{math}
\begin{xy}
(0,0)*{\bullet},+/d10pt/*{ 1},+/u18pt/*{-},
(10,0)*{\bullet},+/d10pt/*{2},+/u18pt/*{+},
(20,0)*{\bullet},+/d10pt/*{3},+/u18pt/*{-},
(30,0)*{\bullet},+/d10pt/*{},
(40,0)*{\bullet},+/d10pt/*{2m-1},+/u18pt/*{-},
(20,10)*{\bullet},+/u10pt/*{0},+/d18pt/*{+},
%
(60,0)*{\bullet},+/d10pt/*{ 1_-},
(70,0)*{\bullet},+/d10pt/*{2_+},
(80,0)*{\bullet},+/d10pt/*{3_-},
(90,0)*{\bullet},+/d10pt/*{4_+},
(100,0)*{\bullet},+/d10pt/*{},
(110,0)*{\bullet},+/d10pt/*{2m_+},
(60,10)*{\bullet},+/l10pt/*{0_+},
(60,20)*{\bullet},+/u10pt/*{ 2m_-},
(70,20)*{\bullet},+/u10pt/*{},
(80,20)*{\bullet},+/u10pt/*{4_-},
(90,20)*{\bullet},+/u10pt/*{3_+},
(100,20)*{\bullet},+/u10pt/*{2_-},
(110,20)*{\bullet},+/u10pt/*{1_+},
(110,10)*{\bullet},+/r12pt/*{0_-},
%
%
(30,-3.5)*{\cdots},
(100,-3.5)*{\cdots},
(70,23.5)*{\cdots},
%
(20,-10)*{\mbox{(a)}},
(80,-10)*{\mbox{(b)}},
\ar (2,0);(8,0)
\ar (18,0);(12,0)
\ar (22,0);(28,0)
\ar (38,0);(32,0)
\ar (2,1);(18,9)
\ar (38,1);(22,9)
\ar (62,0);(68,0)
\ar (78,0);(72,0)
\ar (82,0);(88,0)
\ar (98,0);(92,0)
\ar (102,0);(108,0)
\ar (62,20);(68,20)
\ar (78,20);(72,20)
\ar (82,20);(88,20)
\ar (98,20);(92,20)
\ar (102,20);(108,20)
\ar (60,2);(60,8)
\ar (60,18);(60,12)
\ar (110,8);(110,2)
\ar (110,12);(110,18)
\end{xy}
\end{math}
\caption{(a) Quiver $Q$ for $N=2m$.
(b) Quiver $Q$ for $N=2m+1$.
}
\end{figure}
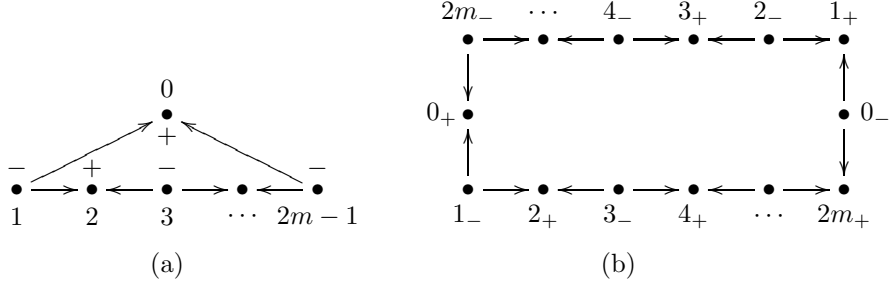

\subsection{Integrability by categorification}
\label{subsec:categorification}

Here we explain, quite briefly, what we mean by
 `integrability
by categorification', which is  mentioned in the introduction,
though we do not use this idea in the rest of the paper.

The categorification of cluster algebras by (generalized)
cluster categories has been recently developed by
several authors
\cite{Caldero06,Buan06, DK,FK, Ke10, A, Ke2,Plamondon10a,Plamondon10b}.
What we  propose is  that one can view such a  categorification also as
a formal method to solve the initial value problem of
the associated T and Y-systems.
Below we concentrate on the case where {\em $I$ is finite
and $B$ is skew-symmetric}.
(For the rest, parallel results in this subsection
are not proved yet, though expected to hold.)

First, let us recall the following
very fundamental fact proved in \cite{FZ07}.

\begin{Theorem}[\cite{FZ07}]
\label{thm:CFG}
 For each seed $(B',x',y')$ of a given
cluster algebra $\mathcal{A}(B,x,y)$,
there exist polynomials $F'_i(y)$  of $y$ $(i\in I)$
and a pair of integer matrices $C'=(c'_{ij})_{i,j\in I}$
and  $G'=(g'_{ij})_{i,j\in I}$ such that the following
formulas hold:
\begin{align}
y'_i &= \left(\prod_{j\in I} y_j^{c'_{ji}}\right)
\prod_{j\in I} F'_j(y_1,\dots,y_n)^{b'_{ji}},
\\ 
\label{x-F}
x'_i &= \left(\prod_{j\in I} y_j^{g'_{ji}}\right)
\frac{F'_i(\hat{y}_1,\dots,\hat{y}_n)}
{F'_i(y_1,\dots,y_n)},\quad
\hat{y}_i = y_i \prod_{j\in I} x_j ^{b_{ji}}.
\end{align}
\end{Theorem}
Furthermore, it is known that each $F'_i(y)$ has the constant term $1$
\cite{Derksen10, Plamondon10b, Nagao10}.
This means that $y'_i$ has the Laurent expansion in $y$
whose leading monomial is given by 
\begin{align}
[y'_i]_{\mathbf{T}} &:= \left(\prod_{j\in I} y_j^{c'_{ji}}\right),
\end{align}
which we call a {\em tropical coefficient\/} (called
a {\em principal coefficient\/} in \cite{FZ07}).
Indeed, $[y'_i]_{\mathbf{T}}$'s  satisfy the same
exchange relation \eqref{y-mutation} 
for $y'_i$'s but replacing $\oplus$ therein
by the one for the tropical semifield $\mathbb{P}_{\mathrm{trop}}(y)$
of $y$:
\begin{align}
  \prod_{i\in I} y_i ^{a_i}
\oplus
\prod_{i\in I} y_i ^{b_i}
:=
\prod_{i\in I} y_i ^{\min (a_i,b_i)}.
\end{align}
Equivalently, in terms of the matrix $C'$, we have the following
recursion relation between seeds $(B',x',y')$ and 
$(B'',x'',y'')=\mu_k (B',x',y')$
with the initial condition $C=\mathbb{I}$ at $t_0$ \cite{FZ07}:
\begin{align}
\label{eq:cmat}
c''_{ji}=
\begin{cases}
-c'_{ji} & i=k\\
c'_{ji}+[-c'_{jk}]_+ b'_{ki}
& i\neq k, b_{ki} \leq 0 \\
c'_{ji}+[c'_{jk}]_+ b'_{ki}
& i\neq k, b_{ki} \geq 0, \\
\end{cases}
\end{align}
where $[x]_+= x$ for $x\geq 0$ and 0 for $x<0$.
It is also known that the matrices $C$ and $G^{T}$, {\it i.e.}, the transpose
of $G$, are inverse to each other \cite{N10}.

Now let us turn to describe the categorification of $\mathcal{A}(B,x,y)$
following the most recent and
general results by Plamondon.
The  presentation here is a minimal one, and we ask the reader
to refer to \cite{Plamondon10a,Plamondon10b} for details.

To the quiver $Q$ corresponding to $B$,
define the {\em principal extension\/}  $\tilde{Q}$ of $Q$
as the quiver obtained from $Q$ by adding a new vertex $i'$
and an arrow $i'\rightarrow i$ for each $i\in I$.
Thus the set of vertices in $\tilde{Q}$ is
given by $\tilde{I}:=I\sqcup I'$ with $I':=\{i'\mid i\in I\}$.
Using some potential
$W$ on $\tilde{Q}$, one can construct a certain triangulated category
 $\mathcal{C}=\mathcal{C}_{(\tilde{Q},W)}$ called
the {\em (generalized) cluster category}.
Furthermore, to each seed $(B',x',y')$ of $\mathcal{A}(B,x,y)$ at
$t'\in \mathbb{T}_I$,
a certain rigid object $T'=\bigoplus_{i\in \tilde{I}} T'_i$ in $\mathcal{C}$ is associated
so that the following properties hold.

\begin{Theorem}[\cite{Plamondon10a,Plamondon10b}]
\label{thm:Pla}
Let $T=\bigoplus_{i\in \tilde{I}} T_i$
 be the rigid object in $\mathcal{C}$ for the initial seed $(B,x,y)$.
Then, we have the following:
\begin{align}
\tilde{Q}' & = \mbox{the quiver for
$\mathrm{End}_{\mathcal{C}}(T')$},\\
c'_{ij}& = - \mathrm{ind}_{T'}(T_i[1])_j
=\mathrm{ind}^{\mathrm{op}}_{T'}(T_i)_j,\\
\label{eq:gT}
g'_{ij}& =  \mathrm{ind}_{T}(T'_j)_i,\\
\label{eq:FGr}
F'_i(y)&=\sum_{e\in (\mathbb{Z}_{\geq 0})^{\tilde{I}}}
\chi(\mathrm{Gr}_e(\mathrm{Hom}_{\mathcal{C}}(T,T'_i[1])))
\prod_{j\in I}y_j^{e_j}.
\end{align}
Here, $\tilde{Q}'$ is the quiver obtained from $\tilde{Q}$
by the mutation sequence from $t$ to $t'$,
$\mathrm{Gr}_e$ is the quiver Grassmannian
with dimension vector $e=(e_j)
\in (\mathbb{Z}_{\geq 0})^{\tilde{I}}$,
and $\chi$ is the Euler number.
\end{Theorem}

As a direct application of Theorems \ref{thm:CFG} and \ref{thm:Pla},
one obtains the following `procedure'
to solve the initial value problem for a general
T and Y-systems occurred in $\mathcal{A}(B,x,y)$.

Let $(B',x',y')$ be the seed at  $t'\in \mathbb{T}_I$.
Suppose that we would like to know the expression $x'$ and $y'$
 in terms of initial variables $x$ and $y$.
This is solved in two steps.

{\bf Step 1.} Calculate the matrix $C'$ by solving the
piece-wise linear recursion relation
\eqref{eq:cmat}. 
Or, equivalently, solve the {\em tropical Y-system},
which is obtained from the Y-system by
replacing $+$ with the tropical $\oplus$ in \eqref{x-F}.
Then, thanks to \eqref{eq:gT}
and the result of \cite{Plamondon10a,
Plamondon10b}, the matrix $G'=(C'^{-1})^T$ uniquely determines
the rigid object $T'$ corresponding to the seed $(B',x',y')$.

{\bf Step 2.} Calculate the polynomials $F'_i(y)$ ($i\in I$)
by \eqref{eq:FGr}. Then, applying Theorem \ref{thm:CFG},
 the problem is solved.

One may regard that this procedure is formal,
 in the sense that  the {\em explicit calculation\/}
of the right hand side of \eqref{eq:FGr} is quite a formidable task,
 in general.
On the other hand, one may also regard that this is already the best possible
 answer one can expect for such a general setting,
 in the sense that the formulas in Theorems
 \ref{thm:CFG} and \ref{thm:Pla}  clearly tell us {\em the intrinsic
meaning of the resulted expressions for  $x'$ and $y'$}.

We leave the comparison of this notion 
with other (more conventional and/or strong) ones
of integrability as a very interesting future problem.

\section{Poisson structure for cluster algebra without coefficients}
\subsection{Setting}

In this section we set $\PP = \{ 1 \}$ (the trivial semifield),
where $1 \cdot 1 = 1 \oplus 1 = 1$.  
Then the exchange relation \eqref{y-mutation} becomes trivial,
and \eqref{x-mutation} reduces to
\begin{align}\label{x-red-mutation}
  x'_i 
  =
  \begin{cases}
    x_i & i \neq k, \\
    \displaystyle{
    \frac{\prod_{j:b_{jk} > 0} x_j^{b_{jk}} 
          + \prod_{j:b_{jk} < 0} x_j^{-b_{jk}}}{x_k}} 
    & i = k. 
  \end{cases}
\end{align}

\subsection{Mutation compatible Poisson bracket}\label{subsec:poissonI}

Fix a skew-symmetrizable matrix $B = (b_{ij})_{i,j \in I}$
and a diagonal matrix $D = \mathrm{diag}(d_i)_{i \in I}$ 
as $DB$ is skew-symmetric.
We say that the matrix $B$ is {\it indecomposable} 
if there is no $I_1, I_2 \neq \emptyset$
such that $I=I_1 \sqcup I_2$ and $b_{ij}=0$ for $i \in I_1$, $j \in I_2$.
In the rest of this section,
we assume that $B$ is indecomposable without losing generalities.

The {\em log-canonical} or {\it quadratic} Poisson bracket
for $\{x_i\}_{i \in I}$ is defined by
\begin{align}\label{x-poisson}
\{x_i , x_j \} = p_{ij} x_i x_j, \qquad p_{ij} \in \Q, 
\qquad p_{ji} = -p_{ij},
\end{align}
from which the skew-symmetry $\{x_i , x_j \} = -\{x_j , x_i \}$ and 
the Jacobi identity:
$$
\{\{x_i , x_j \}, x_k\} + \{\{x_j , x_k \}, x_i\}
+ \{\{x_k , x_i \}, x_j\} = 0
$$ 
follow. 
We study the log-canonical Poisson bracket
for $\{x_i\}_{i \in I}$
which is compatible with the exchange relation \eqref{x-red-mutation}
in the following sense:

\begin{Definition}\label{def:mut-com}\cite{GSV02}
We say a Poisson bracket for $x=\{x_i\}_{i \in I}$
 is {\it mutation compatible}
if, for any $k \in I$, the bracket induced for $x' = \mu_k(x)$
 again has  the log-canonical form
$\{x'_i, x'_j \} = p'_{ij} x'_i x'_j$ with some $p'_{ij} \in \Q$.
\end{Definition}

\begin{Theorem}
\label{prop:x-Poisson:tn}
\par
(i) For a skew-symmetric matrix $P=(p_{ij})_{i,j\in I}$, the corresponding
Poisson bracket \eqref{x-poisson} is mutation compatible
if and only if $PB$ is a diagonal matrix.
\par
(ii) Suppose that $PB$ is a diagonal matrix.
Let $P'=(p'_{ij})_{i,j\in I}$ be the matrix
for the induced bracket $\{x'_i,x'_j\}=p'_{ij} x'_i x'_j$.
Then, $P'$ is given by
\begin{align}
\label{P-mutation}
  p'_{ij} 
  = 
  \begin{cases}
    - p_{ik} + \sum_{l:b_{lk}>0} b_{lk} p_{il} & i \neq j = k,
    \\
    - p_{kj} + \sum_{l:b_{lk}>0} b_{lk} p_{lj} & k = i \neq j,
    \\
    p_{ij} & \text{otherwise}.
  \end{cases}
\end{align} 
\par
(iii) The matrix $P'B'$ is again a diagonal matrix if and only if
$PB=cD$ where $c \in \Q$ is a constant.
In this case, $PB$ is invariant under the mutation,
i.e., $PB = P'B' = cD$.  
\end{Theorem}

\begin{proof}
(i) For $k \in I$, set $x' = \mu_k(x)$.
The Poisson bracket compatible with the mutation $\mu_k$ satisfies
$$
  \{ x'_i, x'_j \} 
  =
  \begin{cases}
    \{ x_i, x'_k \} = p'_{ik} x_i x'_k & i \neq k, j = k, \\
    \{ x'_k, x_j \} = - p'_{jk} x'_k x_j & i = k, j \neq k, \\
    \{ x_i, x_j \} & \text{otherwise (including $i=j=k$)}.
  \end{cases}
$$  
We have a nontrivial condition  when $i \neq k$ and  $j = k$:
\begin{align}\label{Pxx1}
\begin{split}
  \{ x_i, x'_k \} 
  &= 
  \left\{ x_i,~ \frac{\prod_{l:b_{lk} > 0} x_l^{b_{lk}} 
          + \prod_{l:b_{lk} < 0} x_l^{-b_{lk}}}{x_k} \right\}
  \\
  &= - p_{ik} x_i x'_k 
     + \frac{x_i}{x_k} 
       \left( \sum_{l:b_{lk}>0} b_{lk} p_{il}\prod_{l:b_{lk} > 0} x_l^{b_{lk}}
        - \sum_{l:b_{lk}<0} b_{lk} p_{il}\prod_{l:b_{lk} < 0} x_l^{-b_{lk}}
       \right)
  \\
  &=  
  p'_{ik} x_i x'_k.
\end{split} 
\end{align}
Thus we have
\begin{align}\label{Pxx2}
  \sum_{l:b_{lk}>0} b_{lk} p_{il} = - \sum_{l:b_{lk}<0} b_{lk} p_{il},
\end{align}
from which $(P B)_{ik} = 0$ follows for $i \neq k$.
\par
(ii) When $i \neq k$ and $j = k$, from \eqref{Pxx1} and \eqref{Pxx2} we obtain
$$ 
  p'_{ik} = - p_{ik} + \sum_{l:b_{lk}>0} b_{lk} p_{il}.
$$
When $i= k$ and $j \neq k$, we have $p'_{kj} = - p'_{jk}$ since
the matrix $P'$ is skew-symmetric.
For other cases of $i,j$, we have $p'_{i,j} = p_{i,j}$.
Thus \eqref{P-mutation} follows.
\par
(iii)
Assume $PB = \mathrm{diag}(\sigma_i)_{i \in I}$.
By direct calculations using \eqref{B-mutation} and \eqref{P-mutation}
we obtain 
$$
(P'B')_{i,j} = 
\begin{cases}
  \sigma_i \delta_{i,j} & i,j \neq k
  \\
  0 & i \neq k = j
  \\
  \sigma_j [b_{jk}]_+ + \sigma_k [-b_{kj}]_+
  &i = k \neq j
  \\
  \sigma_k & i=k=j.
\end{cases}
$$
Thus $P'B'$ becomes a diagonal matrix if and only if
$\sigma_j b_{jk} = - \sigma_k b_{kj}$,
namely $\sigma_k = c d_k$ for all $k \in I$ with a constant $c \in \Q$.
Then the claim follows.
\end{proof}

\begin{Remark}
In Theorem \ref{prop:x-Poisson:tn},
we need the assumption that $B$ is indecomposable 
only at (iii).  
\end{Remark}

\subsection{Poisson structure}\label{subsec:xxx}

Consider the cluster pattern $t'\mapsto (B',x')$ ($t'\in \mathbb{T}_I$)
 for the cluster algebra $\mathcal{A}(B,x)$
in \S \ref{subsec:def}.
In view of Theorem \ref{prop:x-Poisson:tn}, we endow each seed (without coefficients)
$(B',x')$ at $t'\in \mathbb{T}_I$  with a skew-symmetric rational matrix $P'=(p'_{ij})_{i,j \in I}$
satisfying the following properties:
\par
(i) $P'B' = cD$, where $c \in \Q$ is a constant independent of the seeds.
\par
(ii) For any edge $t'\, \frac{k}{\phantom{aaa}}\,  t''$,
the corresponding matrices $P'$ and $P''$ satisfy the 
exchange relation
\begin{align}
\label{P-mutation:tn}
  p''_{ij} 
  = 
  \begin{cases}
    - p'_{ik} + \sum_{l:b'_{lk}>0} b'_{lk} p'_{il} & i \neq j = k,
    \\
    - p'_{kj} + \sum_{l:b'_{lk}>0} b'_{lk} p'_{lj} & k = i \neq j,
    \\
    p'_{ij} & \text{otherwise}.
  \end{cases}
\end{align} 
We  call the assignment $t'\mapsto (B',x';P')$
a {\em Poisson structure} for $\mathcal{A}(B,x)$,
and also call each matrix $P'$ the {\em Poisson matrix} 
at $t'\in \mathbb{T}_I$.

\begin{Remark}
It follows from Theorem 3.2 that for $(B',x'; P')$ and $(B'',x'';P'') 
$ at
$t'$ and $t''$, if $(B',x')= (B'',x'')$, then $P'=P''$. Therefore,
one may also think that the Poisson matrix $P'$ is attached to the  
seed $(B',x')$.
\end{Remark}

Thus, to construct a Poisson structure,
take any skew-symmetric rational matrix $P$ as $PB = cD$ with $c \in \Q$,
and set it as the Poisson matrix at the initial vertex $t_0$.
Then, the Poisson matrices at the other vertices in $\mathbb{T}_I$  are
uniquely determined from $P$ by the exchange relation \eqref{P-mutation:tn}.

One can describe the Poisson matrices more explicitly.
We continue to assume that $B$ is indecomposable.
We first consider the case when the index set $I$ is finite.

\begin{Theorem}[{{\it cf.} \cite[Theorem 1.4]{GSV02}}] 
\label{th:x-Poisson}
Suppose that the index set $I$ of $B$  is finite.
Let $t'\mapsto (B',x';P')$ ($t'\in \mathbb{T}_I$) be any
Poisson structure for $\mathcal{A}(B,x)$,
and let $P$ be the Poisson matrix at $t_0$.

(i) If $B$ is invertible, then $P$  is given by $P=cDB^{-1}$,
 where $c$ is any rational number. Furthermore,
 $P'=cDB'^{-1}$ holds, where $c$ is the same as above.

(ii) If $B$ is not invertible, then $P$
is given by any skew-symmetric matrix which satisfies $PB=O$.
Furthermore, $P'$ also satisfies $P'B'=O$.
\end{Theorem}
\begin{proof}
Note that $DB^{-1}$ is skew-symmetric because
$BD^{-1}$ is skew symmetric.
Then, the claim  is an immediate consequence of the definition
of Poisson structure and Theorem 3.2.
\end{proof}

\begin{Remark}
The result in Theorem \ref{th:x-Poisson} is quite close to 
\cite[Theorem 1.4]{GSV02}, but the assumption of the two theorems 
are slightly different.
When $B$ is skew-symmetric and invertible,
$P$ gives the Poisson bracket studied in \cite{GSV02,GSV03}.
\end{Remark}

\begin{Example}
The Somos $4$ equation:
\begin{align*}
  s_{n+4} s_n = s_{n+3} s_{n+1} + (s_{n+2})^2
\end{align*}
is a simple example described by the cluster algebra
with a non-invertible matrix $B$ \cite{FordyMarsh09,Hone07}:
\begin{align*}
B =
\begin{pmatrix}
 0 & -1 & 2 & -1 \\
 1 & 0 & -3 & 2 \\
 -2 & 3 & 0 & -1 \\
 1 & -2 & 1 & 0  
\end{pmatrix}.
\end{align*}
The general skew-symmetric solution $P$ to $PB=O$
is unique up to a constant number as 
$$
P = 
\begin{pmatrix}
0 & 1 & 2 & 3 \\
-1 & 0 & 1 & 2\\
-2 & -1 & 0 & 1 \\
-3 & -2 & -1 & 0  
\end{pmatrix},
$$
which appeared in \cite[eq.(2.9)]{Hone07}.
We note that this $P$ gives a unique mutation periodic Poisson bracket for 
$x$ at the same time.
\end{Example}

When $I$ is infinite, the situation is a little more complicated
because, in general, the inverse and the associativity
of matrices are more subtle.
For a pair of matrices $M$ and $N$ with an infinite index set $I$,
we say $M$ is a {\em left inverse\/} of $N$ if $MN=\mathbb{I}$.
Note that a left inverse of $N$ is not necessarily unique when
it exists.

\begin{Theorem}
\label{th:x-Poisson-infinite}
Suppose that the index set $I$ of $B$  is infinite,
and let $B'$, $P$, $P'$ be the same as in Theorem \ref{th:x-Poisson}.
\par
(i) If $B$ has a left inverse $M$ such that $DM$ is skew-symmetric,
then $P$ is given by $P=cDM+R$, where $c$ is any
rational number and $R$ is any skew-symmetric matrix
which satisfies $RB=O$.
Furthermore, $P'=cDM'+R'$, where $c$ is the same as above,
$M'$  is a left inverse of $B'$ such that $DM'$ is skew-symmetric,
 and  $R'$ is a skew-symmetric matrix
which satisfies $R'B'=O$.

\par
(ii) If $B$ does not have any left inverse $M$
 such that $DM$ is skew-symmetric,
then $P$ is given by any skew-symmetric matrix
which satisfies $PB=O$.
Furthermore, $P'$ also satisfies $P'B'=O$.
\end{Theorem}
\begin{proof}
(i) By definition, $P$ is any skew-symmetric matrix which satisfies
$PB=cD$ for some $c$.
If $c\neq 0$, we have
\begin{align*}
PB=cD \ &\Longleftrightarrow \ (D^{-1}P)B = D^{-1}(PB) = c \mathbb{I}\\
&\Longleftrightarrow \
c^{-1}D^{-1}P = M\ (\mbox{$M$: a left inverse of $B$})\\
&\Longleftrightarrow \ P = cDM\ (\mbox{$M$: a left inverse of $B$}).
\end{align*}
All the associativities used here   are easily justified;
for example,
$(D^{-1}P)B = D^{-1}(PB)$ holds because $D^{-1}$ is a diagonal matrix.
Also, if both $M$ and $M'$ are left inverses of $B$,
then $(cDM- cDM')B=O$. Thus, $cDM'=cDM+R$ for some $R$ with $RB=0$.
Thus, we obtain the claim.
\par
(ii) Let $PB=cD$. Then, the assumption
and the argument in (a) show that $c=0$.
Thus, the claim follows.
\end{proof}

\begin{Example}\label{ex:infinite}
The left inverse of
the skew-symmetric matrix $B$ for the infinite quiver $Q$ depicted 
at Figure 3 is not unique.
Here the index set $I$ of $B$ is $\Z$, and 
$B = (b_{ij})_{i,j \in I}$ is given by 
$b_{ij} = (-1)^i (\delta_{i,j+1} + \delta_{i,j-1})$.
The Poisson matrix $P = cM + R$ which satisfies $PB=c\mathbb{I}$ is
given by 
\begin{align*}
&M=(m_{ij})_{i,j \in I}, \\ 
&\qquad
m_{2k,j} = \begin{cases} 0 & j \geq 2k \\
                         \sin \frac{(j-2k)\pi}{2} & j < 2k
           \end{cases}, \quad
m_{2k+1,j} = \begin{cases} \sin \frac{(j-2k-1)\pi}{2} & j > 2k+1 \\
                           0 & j \leq 2k+1 \\
           \end{cases},
\\
&R= a (r_{ij})_{i,j \in I}, \qquad 
r_{ij} = \sin \frac{(j-i)\pi}{2}; ~ a \in \Q,
\end{align*}
where $MB = \mathbb{I}$ and $RB=O$ hold.

\begin{figure}
\label{fig:Q-infinite}
\begin{math}
\begin{xy}
(-12,0)*{\cdots},
(0,0)*{\bullet},+/d10pt/*{ -2},
(10,0)*{\bullet},+/d10pt/*{-1},
(20,0)*{\bullet},+/d10pt/*{0},
(30,0)*{\bullet},+/d10pt/*{1},
(40,0)*{\bullet},+/d10pt/*{2},
(50,0)*{\bullet},+/d10pt/*{3},
(62,0)*{\cdots},
\ar (-2,0);(-8,0)
\ar (2,0);(8,0)
\ar (18,0);(12,0)
\ar (22,0);(28,0)
\ar (38,0);(32,0)
\ar (42,0);(48,0)
\ar (58,0);(52,0)
\end{xy}
\end{math}
\caption{Infinite quiver Q
(Example \ref{ex:infinite})}
\end{figure}
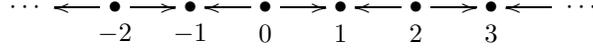

\end{Example}

\subsection{Induced Poisson bracket}

Following \cite{GSV02} we introduce variables $f_i$:
\begin{align}\label{f-x}
  f_i = \prod_{j \in I} x_j^{b_{ji}}, \qquad i \in I. 
\end{align}

\begin{Prop}[{{\it cf.} \cite[Theorem 1.4]{GSV02}}] 
\label{lemma:f-poisson}
The Poisson structure for $\mathcal{A}(B,x)$ induces the Poisson bracket
for $f_i$:
$$
  \{f_i , f_j \} = p_{ij}^f f_i f_j, 
  \quad P^f = (p^f_{ij})_{i,j \in I} = - cDB.
$$
Further, we have 
\begin{align}\label{fx-poisson}
\{f_i, x_j\} = -c \delta_{ij} d_i f_i x_j.
\end{align}
\end{Prop}

\begin{proof}
From the definition of $f_i$, it is easy to see 
\begin{align}\label{P-Pf}
  P^f = B^T P B.
\end{align}
Since $P$ satisfies $PB=cD$, we obtain 
$P^f = B^T c D = - c DB$, which is skew-symmetric.
Further, we see
\begin{align*}
  \{f_i, x_j\} 
  &= \{\prod_{l:b_{li}\neq 0} x_l^{b_{li}}, x_j \}
  = \sum_{l:b_{li}\neq 0} b_{li} f_i x_j p_{lj}
  \\
  &= - (PB)_{ji} f_i x_j 
  = -c \delta_{ij} d_i f_i x_j
\end{align*}
and the claim follows.
\end{proof}

\subsection{Compatible $2$-form}

For a log-canonical Poisson bracket \eqref{x-poisson},
we say that the $2$-form 
$$
  \frac{1}{2} \sum_{i,j \in I} \omega_{ij}
  \frac{dx_i}{x_i} \wedge \frac{dx_j}{x_j}
$$
is compatible with the Poisson bracket if $P W = d \,\mathbb{I}$ holds
with a constant number $d \neq 0$, where $W=(\omega_{ij})_{i,j \in I}$.

When the matrices $B$ and $P$ satisfy $PB=cD$ with $c \neq 0$,
the $2$-form compatible with the Poisson bracket \eqref{x-poisson} 
is given by 
$$
  \omega = \frac{1}{2} \sum_{i,j \in I} b_{ij} d_j^{-1} 
    \frac{dx_i}{x_i} \wedge \frac{dx_j}{x_j}  
$$
up to a constant.
  
\begin{Remark}
The $2$-form invariant under the mutation 
was first introduced in \cite{GSV03}.
The above $2$-form $\omega$ is its generalization.
\end{Remark}

\section{Poisson structure for cluster algebra with coefficients}
\subsection{Setting}

In this section we study the case $\PP$ is a universal semifield 
$\PP_{\mathrm{univ}}(y)$ generated by $y = (y_i)_{i \in I}$,
which is the set of all rational functions of $y_i ~(i \in I)$
written as subtraction-free expressions.
Here the operation $\oplus$ is the usual addition.

\subsection{Mutation compatible Poisson bracket}

Fix a skew-symmetrizable and indecomposable matrix $B = (b_{ij})_{i \in I}$
and a diagonal matrix $D = \mathrm{diag}(d_i)_{i \in I}$ 
as $DB$ is skew-symmetric.
In the following, for the matrix $M$ we write $M^T$ for the transpose
of $M$.
We study the mutation compatible Poisson bracket
for $\{x_i, y_i\}_{i \in I}$ in the sense of 
Definition \ref{def:mut-com}.
We set 
\begin{align}\label{xy-poisson-all}
\{ y_i, y_j \} = p^y_{ij} y_i y_j, 
\quad 
\{ x_i, y_j \} = p^{xy}_{ij} x_i y_j,
\quad  
\{ x_i, x_j \} = p_{ij}^x x_i x_j,
\end{align}
with $p_{ij}^x, p^{xy}_{ij}, p^y_{ij} \in \Q$,
and define a matrix $\mathcal{P}$:
\begin{align}\label{big-P}
\mathcal{P} = \begin{pmatrix}
              P^x & P^{xy} \\ -P^{xy \, T} & P^y
              \end{pmatrix},
\end{align}
where 
$P^x=(p^x_{ij})_{i,j \in I},
~ P^{xy}=(p^{xy}_{ij})_{i,j \in I},
~P^y=(p^y_{ij})_{i,j \in I}$.

The Poisson bracket for $y$ compatible with the mutation 
\eqref{y-mutation}
is uniquely determined up to some constant $c_y \in \Q$ as \cite[\S 2.1]{FockGon07}
\begin{align}\label{yy-poisson}
  P^y = c_y DB.
\end{align}
(The matrix $B$ in \cite{FockGon07} corresponds to $B^T$ here.) 
In view of \eqref{fx-poisson}, we assume 
\begin{align}\label{xy-poisson}
P^{xy} =\mathrm{diag}(p_i)_{i \in I},
\end{align}
and construct $P^x$.

\begin{Prop}\label{prop:xy-poisson}
The mutation compatible Poisson brackets are given by 
\begin{align}\label{xy-xx-poisson}
  P^{xy} = c_y D, \qquad P^x = P,
\end{align}
where $P$ is what obtained in Theorem \ref{th:x-Poisson}
or Theorem \ref{th:x-Poisson-infinite}.
More precisely, when $B$ has the inverse (resp. a left inverse $M$ 
such that $DM$ is skew-symmetric), we have $P=c_x D B^{-1}$
(resp. $P=c_x DM$), where $c_x \in \Q$ is some constant. 
Otherwise, $P$ is any solution to $PB=O$.
\end{Prop}

\begin{proof}
We determine $P^{xy}$ and $P^x$ in this order. 
In this proof we consider the mutation at $k \in I$,
and write $(B',x',y') = \mu_k(B,x,y)$, 
$\mathbb{X}_k^+ = \prod_{l:b_{lk} > 0} x_l^{b_{lk}}$
and $\mathbb{X}_k^- = \prod_{l:b_{lk} < 0} x_l^{-b_{lk}}$.

For $k \neq j$ and $b_{kj}> 0$, we have
\begin{align*}
  \{x_k', y_j' \} 
  &=
  \left\{ \frac{y_k \mathbb{X}_k^+ + \mathbb{X}_k^-}{(1 + y_k)x_k},~
          y_j \left(\frac{y_k}{1 + y_k}\right)^{b_{kj}} \right\}
  \\
  &= \frac{y_k \mathbb{X}_k^+ + \mathbb{X}_k^-}{1 + y_k}
     \left\{ \frac{1}{x_k}, 
            \left(\frac{y_k}{1 + y_k}\right)^{b_{kj}} \right\} y_j
     \\
     & \quad
     + \left(\frac{\mathbb{X}_k^+}{x_k} 
             \left\{\frac{y_k}{1 + y_k}, y_j \right\}
             + \frac{\mathbb{X}_k^-}{x_k} 
             \left\{\frac{1}{1 + y_k}, y_j \right\}
             + \frac{1}{(1 + y_k)x_k} \{ \mathbb{X}_k^-, y_j \}
       \right) \left(\frac{y_k}{1 + y_k}\right)^{b_{kj}}  
     \\
     &= -x_k' b_{kj} p_k y_j' \frac{1}{1+y_k}
     \\ & \quad 
     + \frac{1}{(1+y_k)x_k}
       \left(\mathbb{X}_k^+ d_k b_{kj} c_y \frac{y_k}{1+y_k}
             - \mathbb{X}_k^-d_k b_{kj} c_y
               \frac{y_k}{1+y_k} 
             - \mathbb{X}_k^- b_{jk}p_j
       \right) y_j'.
\end{align*}
This should be zero by \eqref{xy-poisson}.
Thus we obtain $p_k = c_y d_k$ and $p_j = c_y d_j$.
For $b_{kj}< 0$, we similarly obtain $p_k = c_y d_k$ and  
$p_j = c_y d_j$ by calculating $\{x_k', y_j'\} = 0$. 
Then we obtain $P^{xy} = c_y D$.

Due to \eqref{xy-poisson},
the computation of $\{x_i, x_k'\}$ is essentially same as that in
the proof of Theorem \ref{prop:x-Poisson:tn} (i), and we obtain $P^x = P$.
\end{proof}

\subsection{Poisson structure}

Consider the cluster pattern $t'\mapsto (B',x',y')$ ($t' \in \mathbb{T}_I$)
for the cluster algebra $\mathcal{A}(B,x,y)$ in \S \ref{subsec:def}.
In the same manner as \S \ref{subsec:xxx}, we endow
each seed $(B',x',y')$ at $t'\in \mathbb{T}_I$  
with a matrix $\mathcal{P}'$ composed by a skew-symmetric rational matrix 
$P^{x \prime}=(p^{x\prime}_{ij})_{i,j \in I}$
which satisfies the properties (i) and (ii) in \S \ref{subsec:xxx},
and the matrices $P^{y \prime} = c_y DB'$ and $P^{xy \prime} = c_yD$. 
We  call the assignment $t'\mapsto (B',x',y';\mathcal{P}')$
a {\it Poisson structure} for $\mathcal{A}(B,x,y)$,
and call $\mathcal{P}'$ \eqref{big-P}
the {\it Poisson matrix} at $t' \in \mathbb{T}_I$. 

\begin{Remark}
The Poisson algebra $\mathcal{PA}$ on $\Q \PP(u)$ 
generated by \eqref{xy-poisson-all} with \eqref{yy-poisson} and 
\eqref{xy-xx-poisson} is 
a generalization of that introduced in \cite[\S 2.2]{FockGon07}.
We obtain \cite[\S 2.2]{FockGon07} by setting $P^x=O$.
We note that $x$ with $P^x$ and $y$ with $P^y$ respectively
generate the Poisson subalgebras of $\mathcal{PA}$.
\end{Remark}

\subsection{Compatible $2$-form}

We fix $c_y \neq 0$ and $c_x$ as $c_x + c_y \neq 0$ 
in \eqref{yy-poisson} and \eqref{xy-xx-poisson}.
\begin{Lemma}
The $2$-form compatible with the Poisson structure
is given by 
\begin{align*}
  \Omega 
  = \frac{1}{2} \sum_{i,j \in I} b_{ij} d_j^{-1} 
    \frac{dx_i}{x_i} \wedge \frac{dx_j}{x_j} 
    - \sum_{i \in I} d_i^{-1} \frac{dx_i}{x_i} \wedge \frac{dy_i}{y_i}    
    + \frac{1}{2 c_y} \sum_{i,j \in I} d_i^{-1} p_{ij} d_j^{-1} 
      \frac{dy_i}{y_i} \wedge \frac{dy_j}{y_j}. 
\end{align*}
\end{Lemma}

\begin{proof}
Let $\mathcal{W}$ be the matrix
which correspond to the the above $2$-form:  
$$
  \mathcal{W} = \begin{pmatrix} 
                B D^{-1} & - D^{-1} \\ D^{-1} & c_y^{-1}D^{-1} P D^{-1} 
                \end{pmatrix}.
$$
By using $PB = c_xD$, $DB = -B^TD$ and $P^T = -P$,
we obtain $\mathcal{P} \mathcal{W} = (c_x + c_y)\mathbb{I}$. 
Thus the claim follows.
\end{proof}
 
\begin{Remark}
By setting $P=O$ in $\Omega$, 
we obtain the $2$-form studied in \cite[\S 2.2]{FockGon07}.
\end{Remark}

\section{Poisson bracket for difference equations}
\subsection{The discrete LV equation}
\subsubsection{Mutation compatible Poisson bracket}
\label{sec:dLV}

The infinite matrix $B$ \eqref{B-LV} does not have a left inverse
because
$$
  b_{j,3k} + b_{j,3k+1} + b_{j,3k+2} = 0, \quad k,j \in \Z
$$
holds. Thus the Poisson structure for $\mathcal{A}(B,x)$ 
is given by a skew-symmetric solution $P$ to 
$PB=O$ (due to Theorem \ref{th:x-Poisson-infinite} (ii)).
In this subsection we study the general skew-symmetric solution.

Define $3$ by $3$ submatrices of $P$:
$$
P(i,j) =   
\begin{pmatrix}
  p_{3i,3j} & p_{3i,3j+1} & p_{3i,3j+2} \\
  p_{3i+1,3j} & p_{3i+1,3j+1} & p_{3i+1,3j+2} \\
  p_{3i+2,3j} & p_{3i+2,3j+1} & p_{3i+2,3j+2} 
\end{pmatrix},
\quad i,j \in \Z.
$$
Remark that $P(i,i)$ is skew-symmetric and that we
have $P(i,j) = -P(j,i)^T$ for $i \neq j$, due to the skew-symmetry of $P$.
Fix arbitrary two maps $a$ and $b$ from $\Z$ to $\Q$, 
and define a family of $3$ by $3$ diagonal matrices:
$$
\{Q_{i,j} = \mathrm{diag}(q_{i,j},q_{i,j}+a(i)-a(j),q_{i,j}+b(i)-b(j)) ~|~ 
i,j \in \Z; i \neq j; q_{i,j} \in \Q \}.
$$ 
Set a $3$ by $3$ matrix:
$$
S = 
\begin{pmatrix}
  1 & 1 & 1 \\
  1 & 1 & 1 \\
  1 & 1 & 1 
\end{pmatrix}.
$$

\begin{Theorem}\label{thm:LV-poisson}
The general skew-symmetric solution $P$ to $PB=O$ is given by
\begin{align}
  \label{P00}
  &P(0,0) = 
  \begin{pmatrix}
    0 & a_0 & b_0 \\
    -a_0 & 0 & c_0 \\
    -b_0 & -c_0 & 0
  \end{pmatrix}
  &a_0, b_0, c_0 \in \Q,
  \\
  \label{Pii}
  &P(i,i) =  P(0,0) + Q_{0,i} S - S Q_{0,i} & i \neq 0,
  \\
  \label{Pij}
  &P(i,j) = P(0,0) + Q_{i,j} S & i < j,
  \\
  \label{Pji}
  &P(j,i) = -P(i,j)^T & i < j.
\end{align}
\end{Theorem}

\begin{proof}
The equation $PB=O$ is equivalent to
\begin{align}\label{PB=Sigma}
  \begin{split}
  &(PB)_{i,3k} = p_{i,3k-2} - p_{i,3k-1} - p_{i,3k+1} + p_{i,3k+2} 
  = 0,
  \\ 
  &(PB)_{i,3k+1}
  =-p_{i,3k-2} + p_{i,3k-1} +p_{i,3k} - p_{i,3k+2} - p_{i,3k+3}+p_{i,3k+4} = 0,
  \\
  &(PB)_{i,3k+2} = -p_{i,3k} + p_{i,3k+1} + p_{i,3k+3} - p_{i,3k+4} 
  = 0,
  \end{split}
\end{align} 
for all $i,k \in \Z$. 
The equations in \eqref{PB=Sigma} are solved as
\begin{align}\label{P-condition}
  (p_{i,3k}, p_{i,3k+1}, p_{i,3k+2}) 
  = (p_{i,3k}, p_{i,3k}+s_i, p_{i,3k}+t_i),
\end{align}
with some $s_i, t_i \in \Q$ independent of $k$.
We define a family of $3$ by $3$ diagonal matrices:
$$
\{\tilde{Q}_{i,j} = \mathrm{diag}(q_{i,j}^{(1)},q_{i,j}^{(2)},q_{i,j}^{(3)})
~|~ i,j \in \Z; i \neq j; q_{i,j}^{(1)},q_{i,j}^{(2)},q_{i,j}^{(3)} \in \Q \}.
$$ 

The matrix $P(0,0)$ is generally written as \eqref{P00}.
For any $j \in \Z \setminus \{0\}$,
the matrix $P(0,j)$ is determined by \eqref{P-condition} ($i=0,1,2;k=j$) as
$$
P(0,j) = P(0,0) + \tilde{Q}_{0,j} S.
$$
Then we obtain $P(j,0) = -P(0,j)^T$ due to the skew-symmetry of $P$.
When $j > 0$ (resp. $j < 0$),
the skew-symmetric matrix $P(j,j)$ is determined by $P(j,0)$
and \eqref{P-condition} ($i=3j,3j+1,3j+2; k=j+1$ \text{(resp. $k=j-1$)}) as  
$$
P(j,j) = P(0,0) + \tilde{Q}_{0,j} S - S \tilde{Q}_{0,j}. 
$$
By starting with one of these submatrix $P(i,i)$,
for any $j \neq i$ we obtain
$$
P(i,j) = P(i,i) + \tilde{Q}_{i,j} S, \quad 
P(j,i) = -P(i,j)^T, \quad 
P(j,j) = P(i,i) + \tilde{Q}_{i,j} S - S \tilde{Q}_{i,j},
$$ 
in the same manner.
Here the above two expressions of $P(j,j)$ should be compatible:
$$
P(j,j) 
= P(0,0) + \tilde{Q}_{0,j} S - S \tilde{Q}_{0,j}
= P(i,i) + \tilde{Q}_{i,j} S - S \tilde{Q}_{i,j},
$$         
which is equivalent to 
$$
q_{0,j}^{(n)} - q_{0,i}^{(n)} - q_{i,j}^{(n)} 
=
q_{0,j}^{(m)} - q_{0,i}^{(m)} - q_{i,j}^{(m)},
\quad n,m \in \{1,2,3\}.
$$ 
This relation is satisfied by 
$$
q_{i,j}^{(2)} = q_{i,j}^{(1)} + a(i) - a(j),
\quad
q_{i,j}^{(3)} = q_{i,j}^{(1)} + b(i) - b(j),
$$
where $a$ and $b$ are any map from $\Z$ to $\Q$.
Thus we obtain 
$$
\tilde{Q}_{i,j} = 
\mathrm{diag}(q_{i,j}^{(1)},q_{i,j}^{(1)}+a(i)-a(j),q_{i,j}^{(1)}+b(i)-b(j)).
$$
Finally the claim follows.
\end{proof}

\subsubsection{Poisson structure with symmetry}

With any solution $P$ to $PB=O$ in Theorem \ref{P00},
one can associate a Poisson structure for $\mathcal{A}(B,x)$.
Let $P(u)$ ($u\in \mathbb{Z}$) be the corresponding Poisson matrix for
$(B(u), x(u))$, defined through the same mutation sequence \eqref{LV-seq}
with $P(0)=P$.
The corresponding Poisson matrix $\mathcal{P}$ \eqref{big-P} also gives 
the Poisson structure for $\mathcal{A}(B,x,y)$.

In view of the discrete LV equation \eqref{eq:U1} and its bilinear form
\eqref{eq:T1} which are homogeneous for the variables $n$ and $t$, 
it is natural to consider Poisson structures satisfying
the symmetry and the periodicity for mutations:
\begin{align}
\label{eq:p11}
p(u)_{i+3,j+3}&=p(u)_{i,j},\\
\label{eq:p12}
p(u+1)_{i,j}&=p(u)_{i-1,j-1},\\
\label{eq:p13}
p(u+3)_{i,j}&=p(u)_{i,j},
\end{align}
where \eqref{eq:p13} is a consequence of 
\eqref{eq:p11} and \eqref{eq:p12}.
These are the same symmetry and the periodicity as $B(u)$
\eqref{eq:b11}--\eqref{eq:b13}, however,
they are not necessarily satisfied by a general solution $P$ to $PB=O$.
(On the other hand, the Poisson bracket for $y(u)$ automatically 
has these symmetry and periodicity due to \eqref{yy-poisson}.)

\begin{Prop}\label{prop:sym-P}
The matrix $P=P(0)$ yields a Poisson structure for $\mathcal{A}(B,x)$
with the symmetry and the periodicity \eqref{eq:p11}--\eqref{eq:p13}
if and only if $P$ has the following form:
\begin{align}
  \label{Piis}
  &P(i,i) = 
  \begin{pmatrix}
    0 & a_0 & 2a_0 \\
    -a_0 & 0 & a_0 \\
    -2a_0 & -a_0 & 0
  \end{pmatrix}
  & a_0 \in \Q;~ i \in \Z, 
  \\
  \label{P0js}
  &P(i,j) = P(0,0) + q_{j-i} S & q_{j-i} \in \Q;~ i < j, 
  \\
  \label{Pj0s}
  &P(j,i) = -P(i,j)^T & i < j.
\end{align}
\end{Prop}

\begin{proof}
We first show that a skew-symmetric matrix $P$ in
Theorem \ref{P00} satisfies the condition \eqref{eq:p12}
if and only if $P$ has the form
\eqref{Piis}--\eqref{Pj0s}.
Suppose that $P$ satisfies \eqref{eq:p12}.
{}From the exchange relation of $P$ \eqref{P-mutation} and \eqref{eq:p12}
we obtain
\begin{align}
\label{Pi3kss}
&-p_{i,3k} + p_{i,3k-2} + p_{i,3k+2} = p_{i-1, 3k-1} & 
i \not\equiv 0 \text{ mod }3,
\\
&-p_{3k,j} + p_{3k-2,j} + p_{3k+2,j} = p_{3k-1,j-1} & 
j \not\equiv 0 \text{ mod }3,
\\
\label{Pijss}
&p_{i,j} = p_{i-1,j-1} & \text{otherwise}.  
\end{align}
The condition \eqref{Pijss} is written as
$$
p_{3l+2,3k+2} = p_{3l+1,3k+1} = p_{3l,3k} = p_{3l-1,3k-1},
\quad
p_{3l+1,3k+2} = p_{3l,3k+1}.
$$ 
From the first relation, we see that the maps $a$ and $b$ 
are constant maps, {\it i.e.}, 
$Q_{l,k} = \mathrm{diag}(q_{l,k},q_{l,k},q_{l,k})$, 
and that $Q_{l,k} = Q_{l+1,k+1}$. 
Thus we obtain $P(i,i) = P(0,0)$ and \eqref{P0js}
for all $i \in \Z$. From the second relation and \eqref{Pi3kss},
we obtain $a_0=c_0$ and $b_0 = 2a_0$,
and \eqref{Piis} follows.
Conversely, suppose that $P$ has the form
\eqref{Piis}--\eqref{Pj0s}.
Then $(p'_{ij})_{i,j \in I} := \mu_{\overline{0}}(P)$ is obtained
as follows. We have 
\begin{align*}
&p'_{3i+k,3j+l} = p_{3i+k,3j+l} 
= \begin{cases} q_{j-i} = p_{3i+k-1,3i+k-1} & k=l=0,1,2 \\
                q_{i-j} + a_0 = p_{3i,3j+1} & k=1, ~l=2 \\
                q_{i-j} - a_0 =p_{3i+1,3j} & k=2, ~l=1 \\
  \end{cases},
\\
&p'_{3i,3j} = p_{3i,3j} = q_{j-i} = p_{3i-1,3j-1},
\\ 
&p'_{3i+1,3j} = -p_{3i+1,3j} + p_{3i+1,3j-2} + p_{3i+1,3j+2} 
\\
&\qquad = -(q_{j-i} - a_0) + q_{j-1-i} + (q_{j-i} + a_0) 
= q_{j-1-i} + 2 a_0 = p_{3i, 3j-1},
\\
&p'_{3i+2,3j} = -p_{3i+2,3j} + p_{3i+2,3j-2} + p_{3i+2,3j+2} 
\\
&\qquad = -(q_{j-i} - 2 a_0) + (q_{j-1-i}-a_0) + q_{j-i} 
= q_{j-1-i} + a_0 = p_{3i+1, 3j-1}, 
\end{align*}
for $i \leq j$, where we assume $q_0 = 0$. 
In the same way we obtain
$p'_{3i+k,3j+l} = p_{3i+k-1,3j+l-1}$ for $i>j$ and $k,l \in \{0,1,2\}$.
Therefore $p(1)_{i,j}=p(0)_{i-1,j-1}$.
Then, by induction we obtain (5.11).

Once we have \eqref{Piis}--\eqref{Pj0s}, it is satisfied that
\begin{align}\label{3-sym}
  p_{i+3,j+3} = p_{i,j} \quad i,j\in I. 
\end{align}
From \eqref{eq:p12} and \eqref{3-sym}, \eqref{eq:p11} follows. 
\end{proof}

\subsubsection{Periodic case}

For the completeness, we also consider the Poisson structure for the 
quiver $Q$ in Figure 1 with periodic boundary condition.
Namely, we fix a positive integer $m > 2$ and consider 
the $3m$ by $3m$ skew-symmetric matrix 
$\bar{B} = (b_{ij})_{0 \leq i,j \leq 3m-1}$ 
where $b_{ij}$ is given by \eqref{B-LV} 
with the index set $\Z / 3m \Z$.

The $3m$ by $3m$ Poisson matrix $\bar{P}$ for $(\bar{B},x)$
is obtained from \eqref{P00}--\eqref{Pji} by requiring 
$P(i,j) = P(i,j+m) = P(i+m,j)$ for $i,j \in \Z/m\Z$.
Then the maps $a$ and $b$ becomes constant maps, and 
$q_{i,j} = q_{i,j+m} = q_{i+m,j}$ required. Then we obtain the 
following:
\begin{Prop}\label{prop:P-periodic}
(i) The general skew-symmetric solution $\bar{P}$ 
to $\bar{P} \bar{B} =O$ is given by 
\begin{align*}
&P(0,0) = P(i,i) =   \begin{pmatrix}
    0 & a_0 & b_0 \\
    -a_0 & 0 & c_0 \\
    -b_0 & -c_0 & 0
  \end{pmatrix} & a_0,b_0,c_0 \in \Q, 
\\ 
&P(i,j) = P(0,0) + q_{i,j} S & q_{i,j} \in \Q; ~i<j,
\\
&P(j,i) = -P(i,j)^T & i<j,
\end{align*} 
for $0 \leq i,j \leq m-1$.

(ii) The above matrix $\bar{P}$ yields the symmetry and the periodicity 
\eqref{eq:p11}--\eqref{eq:p13} if and only if
$b_0=2a_0$, $c_0=a_0$ and $q_{i,j} = q_{j-i}$ for $0 \leq i < j \leq m-1$.
\end{Prop}

\subsubsection{Poisson bracket for $\hat{y_i}$}

Following \eqref{f-x} 
we define the variable $f_i(u)$ for $(i,u) \in \Z^2$ by 
\begin{align*}
f_i(u) 
= \frac{x_{i-2}(u+1)x_{i+2}(u+2)}{x_{i-1}(u+2)x_{i+1}(u+1)}.
\end{align*}
From Theorem \ref{th:x-Poisson-infinite}(ii) and Lemma \ref{lemma:f-poisson}, 
it is easy to see the following:

\begin{Corollary}\label{cor:f-x}
We have
\begin{align*}
&\{ f_{i}(u) , x_j(u) \} = 0, & i,j,u \in \Z,
\\
&\{ f_i(u) , f_j(v) \} = 0, & i,j,u,v \in \Z.
\end{align*}
\end{Corollary}

The variable $\hat{y}_i(u)$ \eqref{yhat}
is written as $\hat{y}_i(u) = y_i(u) f_i(u)$. 
For the Y-system \eqref{y-rel}, the set of the initial variables in $P_+$
is $\{\hat{y}_{3i+k}(k) ~|~ k=0,1,2,~i \in \Z\}$.
On the other hand, the set of the initial variables for 
the discrete LV equation \eqref{yhat-rel2} is smaller
as $\{\hat{y}_{3i}(0), \hat{y}_{3i+1}(1) ~|~ i \in \Z\}$,
and we give the Poisson brackets for this set: 
\begin{Prop}
Poisson bracket for the set    
$\{\hat{y}_{3i}(0), \hat{y}_{3i+1}(1) ~|~ i \in \Z\}$ is given by
\begin{align*}
&\{ \hat{y}_{3i}(0), \hat{y}_{3j}(0) \} = 
\{ \hat{y}_{3i+1}(1), \hat{y}_{3j+1}(1) \} = 0,
\\
&\{ \hat{y}_{3i}(0), \hat{y}_{3j+1}(1) \} = c_y (-\delta_{j,i} + \delta_{j,i-1})
\hat{y}_{3i}(0) \hat{y}_{3j+1}(1).
\end{align*}
\end{Prop}

\begin{proof}
Due to Proposition \ref{prop:xy-poisson} and 
Corollary \ref{cor:f-x}, the Poisson bracket for 
$\hat{y}_{i}(u)$ is determined by the matrices $B$,
and independent of the Poisson matrix $P$.
Then, the coefficient $p_{ij}^{\hat{y}}(u)$ of 
the Poisson bracket $\{\hat{y}_{i}(u), \hat{y}_{j}(u) \}
= p_{ij}^{\hat{y}}(u) \hat{y}_{i}(u) \hat{y}_{j}(u)$
for $\hat{y}_{i}(u)$ has the same symmetry and periodicity as
\eqref{eq:p11}--\eqref{eq:p13}.

Note that $\{ y_{i}, f_j \} = 0$ if $i \equiv j$ mod $3$.
It is easy to see 
\begin{align*}
\{ \hat{y}_{3i}(0), \hat{y}_{3j}(0) \} 
= \{y_{3i} f_{3i}, y_{3j} f_{3j} \} = 0.
\end{align*}
Thus $\{ \hat{y}_{3i+1}(1), \hat{y}_{3j+1}(1) \} = 0$ follows from 
\eqref{eq:p12}.
Further we have
\begin{align*}
\{ \hat{y}_{3i}(0), \hat{y}_{3j+1}(1) \} 
&= \{y_{3i} f_{3i}, y_{3j+1}' f_{3j+1}(1) \} 
\\
&= \{y_{3i} , y_{3j+1}'\} f_{3i} f_{3j+1}(1)
   + \{y_{3i}, \frac{x'_{3j+3}}{x'_{3j}}\} 
     f_{3i} y_{3j+1}' \frac{x_{3j-1}}{x_{3j+2}}
\\ & \qquad 
   + \{\frac{x_{3i-2}}{x_{3i+1}}, y_{3j+1}' \} 
     y_{3i} \frac{x_{3i+2}}{x_{3i-1}} f_{3j+1}(1)
\\
&= c_y(-\delta_{j,i} + \delta_{j,i-1}) \hat{y}_{3i}(0) \hat{y}_{3j+1}(1),
\end{align*}
where we write $x = \mu_{\overline{0}}(x)$, 
$y' = \mu_{\overline{0}}(y)$ and use  
\begin{align*}
&\{ y_{3i}, x'_{3j}\} 
= 
\{ y_{3i}, \frac{y_{3j} x_{3j-2} x_{3j+2} + x_{3j-1} x_{3j+1}}
                    {(1+y_{3j}) x_{3j}} \}
= c_y \delta_{i,j}  y_{3i} x'_{3j},
\\
&\{ y_{3i}, y'_{3j+1}\} 
= 
\{ y_{3i}, y_{3j+1} \frac{1+y_{3j+3}}{1+y^{-1}_{3j}}\}
= c_y (\delta_{j,i} - \delta_{j,i-1}) y_{3i} y'_{3j+1},
\\
&\{ x_{3i+1}, y'_{3j+1}\}
= \{ x_{3i+1}, y_{3j+1} \frac{1+y_{3j+3}}{1+y^{-1}_{3j}}\}
= c_y \delta_{i,j} x_{3i+1} y'_{3j+1}.    
\end{align*}
\end{proof}

\subsection{The discrete Liouville equation}

When $N=2m$,
we give the Poisson brackets for a set of initial variables 
$\{y_{2k}(0), y_{2k+1}(1) ~|~ k \in \Z/m \Z\}$
for \eqref{Liu-y-mutation} in $P_+$.
From \eqref{yy-poisson}, the Poisson bracket for $y$ is given by
$$
  \{y_{2k}, y_{i}\} = - c_y(\delta_{2k+1,i} + \delta_{2k-1,i}) y_{2k} y_i,
$$
which induces the Poisson bracket:
\begin{align}\label{Liu-y-poisson}
  &\{y_{2k}(0), y_{2j}(0)\} = \{y_{2k+1}(1), y_{2j+1}(1)\} = 0,
  \\
  &\{y_{2k}(0), y_{2j+1}(1)\} = - c_y(\delta_{j,k} + \delta_{j,k-1}) 
  y_{2k}(0) y_{2j+1}(1). 
\end{align}
These are identified with the Poisson bracket for
$\{\chi_{2k,0}, \chi_{2k+1,1} ~|~ k \in \Z/m \Z\}$.

Since the matrix $B$ has the inverse $B^{-1}$,
the Poisson bracket for $x$ is uniquely determined by
the Poisson matrix $P=c_x B^{-1}$ up to some constant number $c_x$
(Theorem \ref{th:x-Poisson} (i)).

\begin{Remark}
We would remark that the Poisson bracket \eqref{Liu-y-poisson}
appeared in \cite[\S 3]{FV99}.
Its quantization was also introduced
in studying quantum integrable models in discrete space-time.
See \cite{FV99,FKV01,Kash08} and the references therein.
\end{Remark}

When $N=2m+1$,
the Poisson bracket for $y$ is given by
\begin{align*}
\{ y_{i_+}, y_{j_-} \} = 
-c_y(\delta_{j,i+1} + \delta_{j,i-1}) y_{i_+} y_{j_-}, 
\quad
\{ y_{i_+}, y_{j_+} \} = \{ y_{i_-}, y_{j_-} \} = 0.
\end{align*}
These induce the Poisson bracket for 
the set of initial variables  
$\{y_{i_+}(0), y_{i_-}(1) ~|~ i \in \Z / (2m+1) \Z \}$ 
for \eqref{Liu-y-mutationb} in $P_+$, as
\begin{align*}
&\{y_{i_+}(0), y_{j_-}(1)\} 
  = -c_y(\delta_{j,i+1} + \delta_{j,i-1}) y_{i_+}(0) y_{i_-}(1),
\\
&\{ y_{i_+}(0), y_{j_+}(0) \} = \{ y_{i_-}(1), y_{j_-}(1) \} = 0.
\end{align*}
These are identified with the Poisson bracket for 
$\{\chi_{i,0}, \chi_{i,1} ~|~ i \in \Z / (2m+1) \Z \}$:
\begin{align*}
\{\chi_{i,0}, \chi_{j,1}\} 
  = -c_y(\delta_{j,i+1} + \delta_{j,i-1}) \chi_{i,0} \chi_{j,1},
\quad
\{\chi_{i,0}, \chi_{j,0}\} = \{\chi_{i,1}, \chi_{j,1}\} = 0.
\end{align*}


\end{document}